\documentclass[12pt]{article}
\usepackage{amsmath, amssymb, amsthm, fullpage}
\usepackage{color}

\newtheorem{theorem}{Theorem}[section]
\newtheorem{lemma}[theorem]{Lemma}

\newtheorem{corollary}[theorem]{Corollary}

\newtheorem{proposition}[theorem]{Proposition}

\theoremstyle{definition}

\newtheorem{definition}[theorem]{Definition}

\theoremstyle{remark}

\newtheorem{problem}[theorem]{Problem}

\newcommand{\N}{{\mathbf{N}}}
\newcommand{\IPR}{\mathrm{IPR}}
\newcommand{\CZF}{\mathbf{CZF}}
\newcommand{\SIA}{\mathrm{SIA}}
\newcommand{\FIS}{\mathrm{FIS}}

\newcommand{\zero}{\mathbf{0}}
\newcommand{\one}{\mathbf{1}}
\renewcommand{\i}{\mathbf{i}}
\renewcommand{\t}{\mathbf{t}}
\newcommand{\bo}[1]{\boldsymbol{#1}}
\newcommand{\Type}{\mathbb{T}}
\newcommand{\arrowtype}[2]{{#2}^{#1}}
\newcommand{\FT}{\mathrm{FT}}
\newcommand{\Func}{\mathrm{Func}}
\newcommand{\REA}{\mathbf{REA}}
\newcommand{\T}{\mathbf{T}}
\newcommand{\IZF}{\mathbf{IZF}}
\newcommand{\MP}{\mathrm{MP}}

\newcommand{\appa}{\mathbf{APP}}
\newcommand{\appo}{\mathrm{App}}
\newcommand{\bpn}{\mathbf{p_0}}
\newcommand{\bpo}{\mathbf{p_1}}
\newcommand{\dfi}{:=}
\newcommand{\bk}{\mathbf{k}}
\newcommand{\bs}{\mathbf{s}}
\newcommand{\bd}{\mathbf{d}}
\newcommand{\bp}{\mathbf{p}}

\newcommand{\lando}{\land}
\newcommand{\then}{\to}
\newcommand{\proves}{\vdash}
\newcommand{\paar}[1]{\langle #1 \rangle}
\newcommand{\first}{1^{\it st}}
\newcommand{\second}{2^{\it nd}}
\newcommand{\po}{\mathcal{P}}
\newcommand{\vns}{\mathtt{V}^*}
\newcommand{\vvv}{\mathtt{V}}
\newcommand{\goa}{\mathfrak{a}}
\newcommand{\gob}{\mathfrak{b}}
\newcommand{\goc}{\mathfrak{c}}
\newcommand{\god}{\mathfrak{d}}
\newcommand{\goe}{\mathfrak{e}}
\newcommand{\gof}{\mathfrak{f}}
\newcommand{\gog}{\mathfrak{g}}
\newcommand{\goh}{\mathfrak{h}}
\newcommand{\goi}{\mathfrak{i}}
\renewcommand{\goh}{\mathfrak{h}}
\newcommand{\goj}{\mathfrak{j}}
\newcommand{\goao}{\goa^{\circ}}
\newcommand{\goas}{\goa^{*}}
\newcommand{\gobo}{\gob^{\circ}}
\newcommand{\gobs}{\gob^{*}}
\newcommand{\goco}{\goc^{\circ}}
\newcommand{\trel}{\Vdash_{rt}}

\newcommand{\pow}{\mathcal{P}}

\newcommand{\RE}{\mathcal{RE}}

\newcommand{\gok}{\mathfrak{k}}
\newcommand{\IP}{\mathrm{IP}}
\newcommand{\AC}{\mathrm{AC}_{\mathrm{FT}}}
\newcommand{\ACR}{\mathrm{AC}_{\mathrm{FT}}\mbox{-rule}}
\begin{document}
\title{The Independence of Premise Rule in Intuitionistic Set Theories}
\author{Takako Nemoto and Michael Rathjen\\
{\small {\it School of Information Science, Japan Advanced Institute of Science and Technology}}
   \\
{\small {\it Nomi, Ishikawa 923-1292, Japan}, t-nemoto@jaist.ac.jp}\\[1ex]
 {\small {\it School of Mathematics,
  University of Leeds}} \\
{\small {\it Leeds, LS2 9JT, UK}, M.Rathjen@leeds.ac.uk}}
\date{}
\maketitle

\begin{abstract} Independence of premise principles play an important role in characterizing the modified realizability and the Dialectica interpretations. In this paper we show that a great many intuitionistic set theories are closed under the corresponding independence of premise rule
for finite types over $\mathbb N$. It is also shown that the existence property (or existential definability property) holds for statements of the form 
$\neg A\to \exists x^{\sigma}F(x^{\sigma})$, where the variable $x^{\sigma}$ ranges over a finite type $\sigma$. This applies in particular to
  Constructive Zermelo-Fraenkel Set Theory ($\CZF$) and Intuitionistic Zermelo-Fraenkel Set Theory ($\IZF$),
 two systems known not to have the general existence property. 

On the technical side, the paper uses the method of realizability with truth from \cite{R05} and \cite{DR2} with the underlying partial combinatory algebra 
($pca$) chosen among the total ones. A particular instance of the latter is provided by the substructure of the graph model formed by the semicomputable subsets of $\mathbb N$, which has the advantage that it 
forms  a set $pca$ even in proof-theoretically weak set theories such as $\CZF$.
\\[1ex]
MSC 03B30 03F05 03F15 03F35 03F35
\end{abstract}

\section{Introduction}
There are (at least) three types of classically valid principles that figure prominently in constructive mathematics: {\em Choice in Finite Types} ($\AC$), {\em Markov's} ($\MP$), and the {\em Independence 
of Premise} ($\IP$) principle. All three are required for a characterization of G\"odel's Dialectica interpretation (see \cite[III.5]{troelstra73}, \cite[11.6]{troelstra.hb})
whereas modified realizability for intuitionistic finite-type arithmetic, $\mathbf{HA}^{\omega}$, is axiomatized by $\AC$ and $\IP$ alone (see 
\cite[III.4]{troelstra73}, \cite[3.7]{troelstra.hb}). To be more precise, we introduce the following schemata.
\begin{eqnarray*} \mathrm{IP}_{\mathrm{ef}} && (A\to \exists x^{\sigma}\,B(x))\to \exists x^{\sigma}(A\to B(x))\\[1ex]
\mathrm{AC}_{\mathrm{FT}} && \forall x^{\sigma}\exists y^{\tau}\,C(x,y)\to \exists z^{\sigma\tau}\forall x^{\sigma}\,C(x,zx)\end{eqnarray*}
where $\sigma$ signifies a finite type, $x^{\sigma}$ varies over $\sigma$, and $A$ is assumed to be $\exists$-free, i.e., $A$ neither contains existential quantifiers nor disjunctions.\footnote{Of course, it is also assumed that $x$ is not a free variable of $A$.}
$\mathbf{HA}^{\omega}$ satisfies the following.
\begin{theorem} With  $\Vdash_{mr}$ signifying modified realizability, we have: \begin{itemize} \item[(i)] $\mathbf{HA}^{\omega}+\AC+\IP_{ef}\vdash A\leftrightarrow \exists x\,(x\Vdash_{mr} A)$. \\[1ex]
\item[(ii)] $\mathbf{HA}^{\omega}+\AC+\IP_{ef}\vdash A \,\Leftrightarrow \,\mathbf{HA}^{\omega}\vdash t\Vdash_{mr} A$  for some term $t$.\end{itemize}
 \end{theorem}
 An important application of modified realizability is that $\mathbf{HA}^{\omega}$ is closed under the independence of premise rule
 for $\exists$-free formula, $\mathrm{IPR}_{\mathrm{ef}}$, and also satisfies explicit definability, $\mathrm{ED}^{\sigma}$.
 
 \begin{theorem}\label{haupt}\begin{itemize}
 \item[(i)] If $\mathbf{HA}^{\omega}\vdash A\to \exists x^{\sigma}B(x)$, then $\mathbf{HA}^{\omega}\vdash \exists x^{\sigma}(A\to B(x))$,
 when $A$ is $\exists$-free.
 \item[(ii)] If $\mathbf{HA}^{\omega}\vdash \exists x^{\sigma}C(x)$, then $\mathbf{HA}^{\omega}\vdash C(t)$ for a suitable term $t$.
 \end{itemize}
 \end{theorem}
 
 This paper shows that results similar to Theorem \ref{haupt} hold for a great  many set theories $T$, including 
   Constructive Zermelo-Fraenkel Set Theory ($\CZF$) and Intuitionistic Zermelo-Fraenkel Set Theory ($\IZF$). 

\begin{theorem}\label{Main} Let $\sigma$ be a finite type on $\mathbb N$. Below $\sigma$ denotes the corresponding set. We also assume that $T$ 
proves the existence of each finite type (as a set).
\begin{itemize}
\item[(i)] If $T\vdash \forall x\,[\neg A(x) \to \exists y\in \sigma\, B(x,y)]$, then 
$$T\vdash \exists y\,\forall x\,[\neg A(x) \to  y\in \sigma \,\wedge\, B(x,y)].$$
\item[(ii)] If $T\vdash \forall x\,[\forall u\,D(u) \to \exists y\in \sigma\, B(y)]$ and $T\vdash \forall u\,[D(u)\,\vee\,\neg D(u)]$, then 
$$T\vdash \exists y\,[\forall u\,D(u) \to y\in \sigma\,\wedge\, B(y)].$$
\item[(iii)] If $T\vdash \exists x\in \sigma\,C(x)$, then $$T\vdash \exists! x\in \sigma\,[C(x)\,\wedge\,E(x)]$$ for some formula $E(x)$.
\end{itemize}
\end{theorem}
It is known that $\IZF$ and $\CZF$ have the numerical existence property (see \cite{R05}), so Theorem \ref{Main}(iii) extends this property to a larger collection of existential formulas. On the other hand it is known by
 work of H.~Friedman and S.~\v{S}\v{c}edrov \cite{frsc} that $\IZF$ does not have  the  general existence property ($\mathrm{EP}$)   while A. Swan \cite{swan} proved that $\mathrm{EP}$ also fails for $\CZF$.
 
 On the technical side, the paper uses the method of realizability with truth from \cite{R05} and \cite{DR2} with the underlying partial combinatory algebra 
($pca$) chosen among the total ones. A particular instance of the latter is provided by the substructure of the graph model formed by the semicomputable subsets of $\mathbb N$, which has the advantage that it 
forms  a set $pca$ even in proof-theoretically weak set theories such as $\CZF$.
 
\section{The system $\CZF$}
\subsection{Axioms of constructive set theories}
The language of constructive Zermelo-Fraenkel Set Theory $\CZF$ is same first order
language as that of classical Zermelo-Fraenkel Set Theory $\mathbf{ZF}$ whose only
non-logical symbol is the binary predicate $\in$.
We use $u$, $v$, $w$, $x$, $y$ and $z$, possiblly with superscripts, for variables in the
language of $\CZF$. 
The logic of $\CZF$ is intuitionistic first order logic with equality.

The axioms of $\CZF$ are as follows:
\begin{description}
 \item[Extensionality:] $\forall x\forall y(\forall z(z\in x\leftrightarrow z\in y)\to x=y)$.
 \item[Pairing:] $\forall x\forall y\exists z(x\in z\land y\in z)$.
 \item[Union:] $\forall x\exists y\forall z( (\exists w\in x)(z\in w)\to z\in y)$. 
 \item[Infinity:] $\exists x\forall y[y\in x\leftrightarrow \forall z(z\in
	    y\leftrightarrow \bot)\lor (\exists w\in x)\forall v(v\in y\leftrightarrow
	    v\in w\lor v=w)]$. 
 \item[Set Induction:] For any formula $\varphi$,
	    $\forall x[(\forall y\in x)\varphi(y)\to \varphi(x)]\to \forall
	    x\varphi(x)$.
 \item[Bounded Separation:] $\exists x\exists y\forall z[z\in y\leftrightarrow z\in x\land
	    \varphi(z)]$, for any {\it bounded} formule $\varphi$. A formula is {\it bounded} or
	    {\it restricted} if it is constructed from prime formulae using $\to$, $\neg$,
	    $\land$, $\lor$, $\forall x\in y$ and $\exists x\in y$ only. 
 \item[Strong Collection:] For any formula $\varphi$,
	    \begin{align*}
	    \forall x[(\forall y\in x)\exists z\varphi(y,z)\to 
	    \exists w[(\forall y\in x)(\exists z\in w)\varphi(y,z)\land (\forall z\in
	    w)(\exists y\in x)\varphi(y,z)]].
	    \end{align*}
 \item[Subset Collection:] For any formula $\varphi$,
	    \begin{multline*}
	    \forall x\forall y\exists z\forall u[(\forall v\in x)(\exists w\in
	     y)\varphi(v,w,u)\to \\
	     (\exists y'\in z)[(\forall v\in x)(\exists w\in y')\varphi(v,w,u)\land(\forall
	     w\in y')(\exists v\in x)\varphi(v,w,u)]].
	    \end{multline*}
\end{description}
In what follows, we shall assume that the language $\CZF$ has constants $\emptyset$
denoting the {\it empty set}, 
$\omega$ denoting the set of von Neumann natural numbers.
One can take the axioms $\forall x(x\in\emptyset\leftrightarrow \bot)$ for $\emptyset$
and $\forall u[u\in\omega\leftrightarrow (u=\emptyset\lor(\exists v\in\omega)\forall
w(w\in u\leftrightarrow w\in v\land w=v))]$ for $\omega$.
We write $x+1$ for $x\cup\{ x\}$ and use $l$, $m$, and $n$ for elements of $\omega$.

\begin{lemma}\label{mathematicalinduction}
 $\CZF$ proves the following Full Mathematical Induction Schema for $\omega$:
 \begin{align*}
  \varphi(\emptyset)\land \forall x\in\omega(\varphi(x)\to \varphi(x+1))\to \forall
  x\in\omega\varphi(x). 
 \end{align*}
\end{lemma}
\begin{proof}
Immediate from Set Induction. 
\end{proof}

\begin{lemma}
 $\CZF$ proves the following Full Iteration Scheme $\FIS$:
 For any class $A$, any class function $F:A\to A$ and a set $x\in A$, there uniquely
 exists a set function $v:\omega\to A$ such that $v(\emptyset)=x$ and $v(u+1)=F(v(u))$.
\end{lemma}
\begin{proof}
 Assume that $A$ is a class, $F$ is a class function from $A$ to $A$, and $x\in A$.
 Define the following $\varphi(y,z)$:
 \begin{multline*}
  \varphi(y,z)\equiv y\in\omega\land z\in \mathrm{func}(y,A)\land (\emptyset\in y\to \langle
  \emptyset,x\rangle\in z)\land \\
  \forall u\in \omega\forall w(\langle u,w\rangle\in z\land u+1\in y\to 
  \langle u+1,F(w)\rangle\in z),
 \end{multline*}
 where $z\in\mathrm{func}(y,A)$ is
 \begin{align*}
  \forall u\in y\exists! w\in A(\langle u,w\rangle\in z)\land z\subseteq y\times A. 
 \end{align*}
 By Lemma \ref{mathematicalinduction}, we have that $(\forall y\in\omega)\exists!
 z\varphi(y,z)$.
 By Strong Collection, we have the desired function $h$.
\end{proof}

We consider also several extensions of $\CZF$.
\begin{description}
 \item[Full Separation] $\exists x\exists y\forall z[z\in y\leftrightarrow z\in x\land
	    \varphi(z)]$, for any formule $\varphi$.
 \item[Powerset] $\forall x\exists y\forall z(z\subseteq x\to z\in y)$.
\end{description}
The system $\CZF+(\mathbf{Full\;separation})+(\mathbf{Powerset})$ is called $\IZF$
(cf. \cite{myhille} or \cite[VIII.1]{beeson}). 
\begin{description}
 \item[MP {(Markov's principle)}] $\forall
	    n\in\omega(\varphi(n)\lor\neg\varphi(n))\land
	    \neg\neg\exists n\in\omega\varphi(n)\to\exists n\in\omega\varphi(n)$.
 \item[$\mathbf{AC}_{\omega}$ {(Axiom of Countable Choice)}] \quad \\
	    $u\in \mathrm{func}(\omega, V)\land \forall n\in\omega \exists x(x\in u(n))\to
	    \exists v(v\in\mathrm{func}(\omega, V)\land \forall n\in\omega(v(n)\in u(n)))$,
	    where $V$ is the class $\{ x: x=x\}$ and $v(n)$ is the unique $y$ such that
	    $\langle n,y\rangle\in v$.
 \item[$\mathbf{DC}$  (Dependent Choices Axiom)] \quad \\
	    $v\subseteq u\times u\land(\forall x\in
	    u)(\exists y\in u) (\langle x,y\rangle\in v)\to$ \\
	    \hfill $(\forall x\in u)\exists
	    w(w\in\mathrm{func}(\omega, u)\land w(0)=x\land (\forall n\in\omega)(\langle
	    w(n), w(n+1)\rangle\in v))$.
 \item[$\mathbf{RDC}$ (Relativised Dependent Choice Axiom)]\quad \\
	    $\forall x(\varphi(x)\to\exists y(\varphi(y)\land \psi(x,y)))\to $ \\
	    \hfill $\forall x[\varphi(x)\to \exists v[v\in\mathrm{func}(\omega,
	    \{y:\varphi(y)\})\land v(0)= x\land 
	    (\forall n\in\omega)\psi(v(n), v(n+1))]]$.
\end{description}

\begin{definition}
 A set $x$ is inhabited if $\exists y(y\in x)$.
 An inhabited set $x$ is {\it regular} if $x$ is transitive, and for every $y\in x$ and a
 set $z\subseteq y\times x$ if $(\forall u\in y)\exists v(\langle u,v\rangle\in z)$, then
 there is a set $w\in x$ such that
 \begin{align*}
  \forall u\in y\exists v\in w(\langle u,v\rangle\in z)\land \forall v\in w\forall u\in
  y(\langle u,v\rangle\in z).
 \end{align*}
 The {\it regular extension axiom} $\REA$ is as follows: {\it Every set is a subset of a
 regular set.}
\end{definition}

\begin{definition}
 A set $x$ is {\it projective} if for any $x$-indexed family $(y_u)_{u\in x}$ of inhabited
 sets $y_u$, there exists a function $v$ with domain $x$ such that $v(u)\in y_u$ for all $u\in x$.
 The {\bf Presentation Axiom}, $\mathbf{PAx}$, is the statement that every
 set is the surjective image of a projective set.
\end{definition}

\subsection{Finite types in $\CZF$}
The type structure $\Type$ is defined inductively, outside of the language of $\CZF$, by
the following two clauses: 
\begin{align*}
 & 0\in \Type  &
 & \text{If $\sigma,\tau\in \Type$,
 then $\arrowtype{\sigma}{\tau}\in\Type$}.
\end{align*}
We use lower Greek symbols $\rho$, $\sigma$ and $\tau$ for variables varies type structure.

In $\CZF$, we can simulate the type structure
by fixing a primitive recursive bijection
$\wp:\omega\to \{\emptyset\}\cup\omega\times\omega$ 
such that $\wp(\emptyset)=\emptyset$ and
$\wp(n)=\langle m, l\rangle\to \max\{ m, l\}<n$.
We use $\wp(\emptyset)=\emptyset$ as the code for $0$ and $\wp(n)$ for 
$\arrowtype{\sigma}{\tau}$ 
when $\wp(n)=\langle m, l\rangle$ and $\wp(m)$ and $\wp(l)$ are codes for 
$\sigma$ and $\tau$, respectively.
We do not distinguish $\sigma$ and $n$ such that $\wp(n)$ is a code for $\sigma$,
if it does not cause any confusion.

For sets $x$, $y$ and $z$, let $x\in\Func(y,z)$ be an abbreviation for
``$x$ is a function from $y$ to $z$'', i.e., 
$x\subseteq y\times z\land \forall u\in y\exists v\in z(\langle y,z\rangle\in x)$.

By $\FIS$, we have the set $\FT$ such that
\begin{align*}
 & \{ \langle 0,x\rangle:x\in\omega\}\subseteq\FT, \quad
 \forall x(x\in\FT\to\exists \sigma\exists y(x=\langle \sigma,y\rangle)),& \\
 & \forall x[\langle\arrowtype{\sigma}{\tau}, x\rangle\in\FT\leftrightarrow
 x\in
 \{ y:\langle \sigma,y\rangle\in\FT\},
 \{ z:\langle \tau, z\rangle\in \FT\})] 
\end{align*}

For each finite type $\rho$, $\sigma$ and $\tau$,
let $\bo{\rho}$, $\bo{\sigma}$ and $\bo{\tau}$ be sets of the elements with the type,
i.e., $\bo{\rho}=\{ x:\langle \rho,x\rangle\in\FT\}$, etc..
Define $\phi(\sigma,x)$ as follows:
  \begin{multline*}
   \phi(\sigma,x)\equiv
   \exists y[y=\mathrm{func}(\sigma+1,V)\land
   \langle 0,\omega\rangle\in y\land \\
   \forall \tau,\rho(\tau^{\rho}\leq \sigma\land
   \langle \tau,u \rangle\in y\land\langle \rho,v\rangle\in y\to
   \langle \rho^{\tau},v^u\rangle\in y)\land \langle \sigma,x\rangle\in y].
  \end{multline*}
  \begin{lemma}[$\CZF$]\label{type}
 For each finite type $\sigma$, we have the following:
 that
\begin{align*}
 & \exists x(x=\bo{\sigma}\leftrightarrow \phi({\sigma},x)) &
 & \text{and} &
 & \forall x\forall y(\phi({\sigma},x)\land\phi({\sigma},y)\to x=y).
\end{align*}
  \end{lemma}
\begin{proof}
 This is proved by induction on finite types.
 For $\sigma=0$, then it is clear from
 $\omega=\{ y:\langle \sigma, y\rangle\in\FT\}=\bo{\sigma}$ and
 $\forall x\forall y(\psi(\sigma, x)\land\phi(\sigma, y)\to x=y)$.
The induction step is easy, since, for any $x$ and $y$, the function space 
 $y^{x}=\{ z:z\in\Func(x,y)\}$ exists uniquely in $\CZF$.
\end{proof}
\section{Applicative structure}
In order to define a realizability interpretation we must have a
notion of realizing functions on hand. A particularly  general and
elegant approach to realizability  builds on structures which have
been variably called {\em partial combinatory algebras}, {\em
applicative structures}, or {\em Sch\"onfinkel algebras}. These
structures are best described as the models of a theory $\appa$.
The following presents the main features of $\appa$; for full
details cf. \cite{f75,f79,beeson,troelstra}. The language of
$\appa$ is a first-order language with a ternary relation symbol
$\appo$, a unary relation symbol $N$ (for a copy of the natural
numbers) and equality, $=$, as primitives. The language has an
infinite collection of variables, denoted by $a$, $b$, $c$, $\ldots$, $g$, $h$ and $i$,
and nine distinguished constants:
${\mathbf 0}, {\mathbf s}_N,{\mathbf p}_N, {\mathbf k}, {\mathbf s},
{\mathbf d}, {\mathbf p},\bpn,\bpo$ for, respectively, zero, successor on $N$,
predecessor on $N$, the two basic combinators, definition by
cases, pairing and the corresponding two projections. There is no
arity associated with the various constants.
The {\it terms} of $\appa$ are just the variables and constants, which are denoted by $p$,
$q$, $r$, $s$
and $t$ possibly with sub- and superscripts.
We write $tt'\simeq s$ for $\appo(t,t',s)$.

Formulas are then generated  from atomic ones  using the
propositional connectives and the quantifiers.

In order to facilitate the formulation of the axioms, the language
of $\appa$  is expanded definitionally with the symbol $\simeq$
and the auxiliary notion of an {\it application term} is
introduced. We use $p$, $q$,$\ldots$ $t$ also for application terms.
The set of application terms is given by two clauses:
\begin{enumerate}
  \item {all terms of  $\appa$  are application terms; and}
  \item if  {$ s$} and  $ t$ are application terms,
then  $( st)$  is an application term.
\end{enumerate}
For  $s$ and  $t$ application terms, we have auxiliary, defined
formulae of the form:
\begin{eqnarray*} s \simeq  t& \dfi &\forall a ( s \simeq a
\leftrightarrow  t \simeq a),\end{eqnarray*} if $t$ is not a
variable. Here $ s\simeq a$   (for $a$ a free variable) is
inductively defined by:
$${ s} \simeq { a}\; {\;\mbox{ is }\;}
\left\{ \begin{array}{l@{}l}
{ s} = { a},& \text{if $ s$ is a term of $\appa$,}\\
\exists x,y [{s}_{ 1} \simeq x \,\wedge\,{ s}_{ 2} \simeq y
\,\wedge\, {\appo}(x,y,{ a})] & \text{if $ s$ is of the form
$(s_1s_2)$.}\end{array}\right.$$ Some abbreviations are ${ t}_{
1}t_2\ldots { t}_{ n}$ for ((...(${ t}_{1}{ t}_{ 2}$)...)${ t}_{
n}$); $t\downarrow$ for $\exists a({ t}\simeq a)$ and $\phi ({
t})$ for $\exists a({ t}\simeq a \wedge \phi (a))$.

Some further conventions are useful. Systematic notation for
$n$-{\it tuples\/} is introduced as follows: $(t)$ is $t$ and $(s,t)$
is $\bp st$, $(t_1,\ldots,t_n)$ is defined by
$((t_1,\ldots,t_{n-1}),t_n)$.
For projections, we write $(t)_{i}$ for $\bp_it$ and $(t)_{ij}$ for
$\bp_j(\bp_it)$ for $i,j\in\{ 0,1\}$.
 In this paper, the {\bf logic} of
$\appa$ is assumed to be that of intuitionistic predicate logic
with identity.  $\appa$'s {\bf non-logical axioms} are the
following:
\paragraph{ Applicative Axioms}
\begin{enumerate}
 \item $\appo(a,b,c_1)\lando\appo(a,b,c_2)\,\then\, c_1=c_2$.

 \item $(\bk ab)\downarrow\lando \bk ab\simeq a$.

  \item $(\bs
ab)\downarrow\lando \bs abc\simeq ac(bc)$.

 \item $(\bp
a_0a_1)\downarrow\lando(\bpn a)\downarrow\lando
(\bpo a)\downarrow\lando \mathbf p_{\mathbf i}
(\bp a_0a_1)\simeq a_i$ for $i=0,1$.

\item $ N(c_1) \lando N(c_2)\lando c_1= c_2 \then \bd
abc_1c_2\downarrow\lando \bd abc_1c_2\simeq a$.

\item $ N(c_1) \lando N(c_2)\lando c_1\ne c_2 \then \bd
abc_1c_2\downarrow\lando \bd abc_1c_2\simeq b$.

\item $\forall a\,\bigl(N(a)\then \bigl[{\mathbf s}_Na\downarrow
\lando  {\mathbf s}_Na \ne {\mathbf 0}\lando N({\mathbf
s}_Na)\bigr]\bigr)$.

\item $N({\mathbf 0})\;\wedge\; \forall a\,\bigl(N(a)\lando
x\ne{\mathbf 0}\then \bigl[{\mathbf p}_Na\downarrow \lando
{\mathbf s}_N({\mathbf p}_Na)=a \bigr]\bigr)$.

\item $\forall a\,\bigl[N(a)\then  {\mathbf p}_N({\mathbf
s}_Na)=a\bigr]$.
 \item $\varphi({\mathbf 0})\lando\forall
a\bigl[N(a)\lando\varphi(a)\then \varphi({\mathbf
s}_Na)\bigr]\,\then\,\forall a\bigl[N(a)\then\varphi(a)\bigr]$.
 \end{enumerate}
Let ${\mathbf 1}:={\mathbf s}_N\,{\mathbf 0}$. The applicative
axioms entail that ${\mathbf 1}$ is an application term that
evaluates to an object falling under $N$ but distinct from
$\mathbf 0$, i.e., ${\mathbf 1}\downarrow$, $N({\mathbf 1})$ and $
{\mathbf 0}\ne{\mathbf 1}$.

Employing the axioms for the combinators $\mathbf k$ and $\mathbf
s$ one can deduce an abstraction lemma yielding $\lambda$-terms of
one argument. This can be generalized using $n$--tuples and
projections.
\begin{lemma}\label{Abstraction}{\em (cf. \cite{f75})}
({\em \bf Abstraction Lemma}) For each application term $t$ there
is a new application term $t^*$ such that the parameters of $t^*$
are among the parameters of $t$ minus $a_1\ldots a_n$ and such that
$$\appa\proves t^*\!\downarrow\;\wedge\,
\;t^*a_1\ldots a_n\simeq t.$$

$\lambda(a_1\ldots a_n).t$ is written for $t^*$.
\end{lemma}

The most important consequence of the Abstraction Lemma is the
Recursion Theorem. It can be derived in the same way as for the
$\lambda$--calculus (cf. \cite{f75}, \cite{f79}, \cite{beeson},
VI.2.7). Actually, one can prove a uniform version of the
following in $\appa$.

\begin{corollary}\label{Recursion}{\em\bf (Recursion Theorem)}
$$\forall f\exists g\forall a_1\ldots\forall a_n\,g(a_1,\ldots,a_n)\simeq
f(g,a_1,\ldots,a_n).$$ \end{corollary}

\section{Graph models in constructive set theories}\label{graphmodel}

In this section, we present an easiest example of a model of total applicative structure
constructed in constructive set theories without the powerset axioms, such as $\CZF$.

Plotkin and Scott independently developed a PCA whose universe is
the power set of the integers, $\pow(\mathbb{N})$.
This construction
 exploits the fact that finite subsets of $\mathbb{N}$ can be coded as
 integers; the finite set $\{k_0,\ldots,k_r\}$ with $k_0<\cdots <k_r$ can be coded by
 \begin{eqnarray} [k_0,\ldots,k_r]:=\sum_{i=0}^r2^{k_i}, &&
 [\emptyset]:=0\\
 e_n :=\{k_0,\ldots,k_r\}&\mbox{ iff }& n=[k_0,\ldots,k_r].
\end{eqnarray}
We use $X,Y,Z$ for arbitrary subsets of $\mathbb{N}$. Since the coding of
finite sets is onto $\mathbb{N}$, we can use integer variables for finite
sets. We shall often neglect the distinction between finite sets
and their codes in our notation, and thus write, e.g.,
$$ n\subseteq Y\;:=\;e_n\subseteq Y, \phantom{AA} n\in m\;:=\;n\in e_m; \phantom{AA} n\subseteq m\::=\;e_n\subseteq e_m.$$
 We take $\pi:\mathbb{N}\times\mathbb{N}\to\mathbb{N}$ to be a standard primitive
 recursive pairing function with projections $\pi_0$ and $\pi_1$;
 i.e., $\pi_i(\pi(n_0,n_1))=n_i$ for $i=0,1$; for $\pi(n,m)$ we
 also use the abbreviation $(n,m)$.

 Then, the application is defined as follows
 $$X\cdot Y\;\simeq\;\{n\in\mathbb{N}|\,\exists k\subseteq Y\,(k,n)\in X\}.$$
The whole construction of graph models can be found in \cite[IV.7.5]{beeson}, for example,
or can be done as described in Proposition \ref{POAPP} below.
An important aspect of this model is the application defined {\it totally}, i.e., for each
$X$, $Y\subseteq \mathbb{N}$, there is $Z$ with $X\cdot Y\simeq Z$.
In such a model, we prefer to use $=$ instead of $\simeq$.

In $\CZF$, we have no powerset $\pow(\omega)$ and so we cannot simulate the above
construction.
As it was mentioned in \cite[IV.7.5]{beeson}, the set of semicomputable or recursively
enumerable subsets of $\mathbb{N}$ also forms a model of $\appa$,
which we can construct in $\CZF$ as follows:
Using the canonical interpretation of the language of the first order arithmetic $L_1$ into the 
language of set theory $L_s$, we can regard a formula in $L_1$ as a one in $L_s$.
Let $T$ be the Kleene's $T$-predicate.
Then, for each $n$,  Bounded Separation yields the set
$\{ x\in\omega:(\exists y\in\omega)Tnxy\}$ and Strong Collection yields the set
$\mathcal{RE}(\omega)$ of recursively
enumerable sets of natural numbers, i.e.,
\begin{align*}
\mathcal{RE}(\omega)=\{\{ x\in\omega:(\exists y\in\omega)Tnxy\}:n\in\omega\}. 
\end{align*}
Furthermore, we have the graph
\begin{align*}
 & \{ \langle \langle X,
Y\rangle,Z\rangle|X, Y, Z\in\mathcal{RE}(\omega)\land Z=\{ x\in\omega|\exists
y\in\omega(y\subseteq Y(y,x)\in X)\}\}
\end{align*}
of application defined above is a set again by Bounded Separation.
The next proposition ensures that $\mathcal{RE}(\omega)$ actually forms a model os $\appa$ in
$\CZF$. 
 \begin{proposition}\label{POAPP}
There is an interpretation $(\text{-})^{\dag}$ from formulae the language $L_{\appa}$ of
  $\appa$ into the ones in $L_s$ such that $(\bot)^{\dag}\equiv\bot$ and that
 $\appa\vdash \varphi(\vec{x})$ implies
  $\CZF\vdash\forall\vec{x}\in\mathcal{RE}(\omega)(\varphi(\vec{x}))^{\dag}$ for each
  formula $\varphi(\vec{x})$ with only displayed free variables.
 \end{proposition}
  \begin{proof}
   Fix a bijection $^{\dag}$ from the variables in $L_{\appa}$ into the ones in $L_{s}$.
 We extend $^{\dag}$ to an interpretation from the formulae in $L_{\appa}$ into the ones
   in $L_{s}$.
 First, we set the interpretation of constants in $\appa$.
Define $\bk^{\dag}$, $\bs^{\dag}$, $\bp^{\dag}$, $\bp_0$, $\bp_1$, $\bd$,
   $\bs^{\dag}_{N}$, $\bp_{N}$ and $\mathbf{0}$ as follows:
   \begin{align*}
    & \bk^{\dag}=\{(x,(y,z))|z\in x\}; &
    & \bs^{\dag}\,=\,\{(x,(y,(z,w)))|\exists a\,Q(a,x,y,w,z)\}; & \\
    & \multispan3{$
    \bp^{\dag}\,=\,
    \{(2^n,(y,2n))|n\in\omega\}\,\cup\,\{(x,(2^m,2m+1))|m\in\omega\};$\hfill} & \\
    & \bp_0^{\dag}=\{(2^{2n},n)|n\in\omega\}; & 
    & \bp_1^{\dag}\{(2^{2n+1},n)|n\in\omega\}; & \\
    & \multispan3{$\bd^{\dag}\,=\,\{(y,(x,(2^m,(2^n,z))))|\,(m=n\,\wedge\;z\in
x)\;\vee\;(m\ne n\,\wedge\,z\in y\};$}& \\
    & \bs_{N}^{\dag}\,=\, \{ (2^{n}, 2^{n+1})|n\in\omega\}; &
    & \bp_{N}^{\dag}\,=\, \{ (2^{n+1}, 2^{n})|n\in\omega\}; & \\
    & \mathbf{0}^{\dag}\, =\, \{ \emptyset\},
   \end{align*}
   where \begin{eqnarray}\label{Q}Q(a,x,y,w,z)&\mbox{ iff }&\exists
z_1\subseteq
  z\;(z_1,(a,w))\in x\;\wedge\;\forall c\in a\,\exists z_2\subseteq
  z\,(z_2,c)\in y.\end{eqnarray}
   Then, for an application term $ts$, define $(ts)^{\dag}\equiv \{ x\in\omega:(\exists
   y\subseteq s^{\dag})(y,x)\in t^{\dag}\}$.

   For atomic formulae in $L_{\appa}$,
   set $(\bot)^{\dag}\equiv\bot$,
   $(N(x))^{\dag}\equiv \exists
   n\in\omega(x^{\dag}=\{ n\})$,
   $(\appo(a,b,c))^{\dag}\equiv (ab)^{\dag}=c^{\dag}$ and
   $(t=s)^{\dag}\equiv t^{\dag}=s^{\dag}$, respectively.
   For compound formulae in $L_{\appa}$, define $^{\dag}$ inductively by
   \begin{align*}
   & (\varphi\circ\psi)^{\dag}\equiv (\varphi)^{\dag}\circ(\psi)^{\dag} \;
    \text{for $\circ\in\{ \land, \lor, \to\}$}; & 
    & (\mathbf{Q} x\varphi_0(x))\equiv \exists
    x^{\dag}\in\mathcal{RE}(\omega)(\varphi(x))^{\dag};
    \text{for $\mathbf{Q}\in\{ \exists, \forall\}$}.
    & 
   \end{align*}

   Now it is enough to show that $\CZF\vdash\varphi^{\dag}$ for each universal closure
   $\varphi$ of the axioms of $\appa$. 
   Let $X$, $Y$ and $Z$ be subsets of $\omega$.
   For $\bk$, we have $\bk^{\dag}X=\{(y,z)|z\in X\}$, so
  $\bk^{\dag}XY=\{z|z\in X\}=X$; thus verifying the axioms for the
  combinator $\bk$.
   For $\bs$, we have
\begin{eqnarray*} \bs^{\dag}X &=&\{(y,(z,w))|\exists x\subseteq X\,\exists
a\,Q(a,x,y,w,z)\}\\
 \bs^{\dag}XY&=&\{(z,w)|\exists x\subseteq X\,\exists y\subseteq Y\,\exists
a\,Q(a,x,y,w,z)\}\\
 \bs^{\dag}XYZ&=&\{w|\exists x\subseteq X\,\exists y\subseteq Y\,\exists z\subseteq Z\,\exists
a\,Q(a,x,y,w,z)\}\\
XZ(YZ)&=&\{w|\exists a\subseteq YZ\,(a,w)\in XZ\}.\end{eqnarray*}
From the last two set equations one easily computes that
$\bs^{\dag}XYZ=XZ(YZ)$, so that $\bs^{\dag}$ satisfies the axiom for $\bs$.
For other constants, it is easy to show that each of them has the desired properties by 
noticing that $2^n$ codes the set $\{n\}$. 
  \end{proof}

\section{The general realizability structure}
Realizability semantics are ubiquitous in the study of intuitionistic
theories. In the case of set theory, they differ in important aspects
from Kleene's \cite{kl} realizability in their treatment of the quantifiers.
Its origin is Kreisel's and Troelstra's \cite{kt} definition of
realizability for second order Heyting arithmetic.
The latter was applied to
systems of higher order arithmetic and (intensional) set theory by
Friedman \cite{frieda} and Beeson \cite{beeson}. McCarty
\cite{mc84} and \cite{mc86} adapted Kreisel-Troelstra realizability directly to the
extensional intuitionistic set theories such as $\IZF$.  This type of
realizability can also be formalized in $\CZF$ (see \cite{R04}) to yield a
self-validating semantics for $\CZF$.
\cite{R05} introduced the general realizability structure with truth over an
arbitrary (set) model $\mathcal{A}$ of  $\appa$.

In \cite{R05}, the general realizability structure over $\omega$ as a model of $\appa$ is
defined.
In this paper, we define it over arbitrary models $\mathcal{A}$ of $\appa$ such that
both $|\mathcal{A}|$ and the graph $\{ (x, y, z)\in|\mathcal{A}|^3: \appo(x,y,z)\}$ are
sets, such as $\mathcal{RE}(\omega)$ defined in the previous section.
If $z$ is an ordered pair, i.e., $z=\paar{x,y}$ for some sets
$x,y$, then we use $\first(z)$ and $\second(z)$ to denote the
first and second projection of $z$, respectively; that is,
$\first(z)=x$ and $\second(z)=y$.
\begin{definition}\label{grs}
Ordinals are transitive sets whose elements are transitive also.
As per usual, we use lower case Greek letters $\alpha$ and $\beta$ to range over
ordinals.
 Let $\mathcal{A}$ be a model of $\appa$ such that $|\mathcal{A}|$ is a set.
 Besides $\vvv_{\alpha}$ and $\vvv$,
 we define $\vns_{\mathcal{A},\alpha}$ and $\vns_{\mathcal{A}}$ as follows:
\begin{align}
 \label{eg} \vns_{\mathcal{A}, \alpha} &
 = \bigcup_{\beta\in\alpha}\bigl
 \{\paar{x,y}:\;x\in\vvv_{\beta};\;y\subseteq|\mathcal{A}|\times\vns_{\mathcal{A},
 \beta};\;
(\forall z\in y)\,\first(\second(z))\in x\bigr\}
\\ \nonumber\vvv_{\alpha}&= \bigcup_{\beta\in
\alpha}\po(\vvv_{\beta}) \\
 \nonumber\vns_{\mathcal{A}} &=
\bigcup_{\alpha}\vns_{\mathcal{A},\alpha}\\
\nonumber\vvv &= \bigcup_{\alpha}\vvv_{\alpha}. 
\end{align}
\end{definition}
As the power set operation  is not available in $\CZF$ it is not
clear whether the classes $\vvv$ and $\vns_{\mathcal{A}}$ can be formalized in
$\CZF$. However, employing the fact that $\CZF$ accommodates
inductively defined classes this can be demonstrated in the same
vein as in \cite{R04}, Lemma 3.4.

The definition of $\vns_{\mathcal{A},\alpha}$ in (\ref{eg}) is perhaps a bit
involved. Note first that all the elements of $\vns_{\mathcal{A}}$ are ordered
pairs $\paar{x,y}$ such that $y\subseteq |\mathcal{A}|\times \vns_{\mathcal{A}}$. For
an ordered pair $\paar{x,y}$ to enter $\vns_{\mathcal{A},\alpha}$ the first
conditions to be met are that $x\in\vvv_{\beta}$ and
$y\subseteq|\mathcal{A}|\times\vns_{\mathcal{A}, \beta}$ for some $\beta\in\alpha$.
Furthermore, it is required that $x$ contains enough elements from
the transitive closure of $x$ in that whenever $\paar{u,v}\in y$
then $\first(u)\in x$.

\begin{lemma}\label{stets} $(\CZF)$. \begin{itemize}
\item[(i)] {\em $\vvv$ and $\vns_{\mathcal{A}}$} are cumulative: for
$\beta\in\alpha$, {\em $\vvv_{\beta}\subseteq\vvv_{\alpha}$ and
$\vns_{\mathcal{A}, \beta}\subseteq \vns_{\mathcal{A}, \alpha}$.}
 \item[(ii)] For all sets $x$,  {\em $x\in \vvv$}.
 \item[(iii)]  If $x,y$ are sets,
 {\em $y\subseteq |\mathcal{A}|\times\vns_{\mathcal{A}}$} and
	       $(\forall z\in y)\,\first(\second(z))\in x$,
	       then {\em $\paar{x,y}\in\vns_{\mathcal{A}}$}.
 \end{itemize}
\end{lemma}
\begin{proof}
 This is proved in the same way as \cite[Lemma 4.2]{R05}.
\end{proof}

\section{Defining realizability}

We now proceed to define a notion of realizability over $\vns_{\mathcal{A}}$. We
use lower case gothic letters
$\goa$, $\gob$, $\ldots$, $\gok$, possibly with superscripts
as variables to range over elements of $\vns_{\mathcal{A}}$ while variables
$a$, $b$, $\ldots$, $j$ will be reserved for elements of $|\mathcal{A}|$. Each
element $\goa$ of $\vns_{\mathcal{A}}$ is an ordered pair $\paar{x,y}$, where
$x\in\vvv$ and $y\subseteq\mathcal{A}\times\vns_{\mathcal{A}}$; and we  define the
components of $\goa$ by
\begin{eqnarray*}\goao &:=& \first(\goa)=x\\
\goas &:=& \second(\goa)=y.
\end{eqnarray*}

\begin{lemma}\label{fakt} For every $\goa\in\vns_{\mathcal{A}}$, if
$\paar{e,\goc}\in\goas$ then $\goco\in\goao$.\end{lemma}
 \begin{proof}
  This is immediate by the definition of $\vns_{\mathcal{A}}$.
 \end{proof} 
 If $\varphi$ is a sentence with parameters in
$\vns_{\mathcal{A}}$, then $\varphi^{\circ}$ denotes the formula obtained from
$\varphi$ by replacing each parameter $\goa$ in $\varphi$ with
$\goao$.

\begin{definition}\label{real}
We define $e\trel \phi$ for elements $e\in|\mathcal{A}|$ and sentences $\phi$ with
 parameters in 
$\vns_{\mathcal{A}}$. 
 Bounded quantifiers will be treated as quantifiers in their own right, i.e., bounded
and unbounded quantifiers are treated as syntactically different kinds of quantifiers.
 (The subscript $_{rt}$ is supposed to serve as a reminder
of ``realizability with truth''.)
 \begin{eqnarray*}
  e\trel \goa\in \gob &\mbox{iff} &
\goao\in\gobo\;\,\wedge\;\,\exists\,
\goc\;\bigl[\paar{(e)_0,\goc}\in \gobs\;\wedge\; (e)_1\trel
\goa=\goc\bigr]
\\
e\trel \goa=\gob &\mbox{iff}& \goao=\gobo\;\, \wedge\;\, 
\forall
f\forall\goc\,\bigl[\paar{f,\goc}\in \goas 
\;\rightarrow\; (e)_0f\trel \goc\in \gob\bigr] 
 \\
&& \qquad\qquad\wedge\;\,\forall
f\forall\goc\,\bigl[\paar{f,\goc}\in \gobs
\;\rightarrow\; (e)_1 f\trel \goc\in \goa\bigr] \\
e\trel \phi\wedge\psi &\mbox{iff}& (e)_0\trel \phi\;\wedge\;(e)_1\trel \psi\\
e\trel \phi\vee\psi &\mbox{iff}& \bigl[(e)_0=\mathbf{0}\,\wedge\,(e)_1\trel
\phi\bigr]\;\vee
\;\bigl[(e)_0=\mathbf{1}\;\wedge\;(e)_1\trel \psi\bigr]\\
e\trel\neg\phi \phantom{AA} &\mbox{iff}& \neg\phi^{\circ}\;\,\wedge\;\forall f\;\neg f\trel \phi\\
e\trel \phi\rightarrow\psi &\mbox{iff}&
(\phi^{\circ}\rightarrow\psi^{\circ})\;\,\wedge\;\,\forall f\,
\bigl[f\trel \phi\;\rightarrow\;
ef\trel \psi \bigr]\\   
e\trel (\forall x\in \goa)\; \phi  &\mbox{iff}& (\forall
x\in\goao)\phi^{\circ} \;\,\wedge\\
&& \forall f\,\forall \gob\bigl(\paar{f,\gob}\in\goas
 \;\rightarrow\; ef\trel \phi[x/\gob]\bigr)\\
e\trel (\exists x\in \goa)\phi  &\mbox{iff}& \exists \gob\,\bigl(
\paar{(e)_0,\gob}\in \goas\;\wedge\; 
(e)_1\trel \phi[x/\gob]
\bigr) \\
e\trel \forall x \phi \phantom{Ai} &\mbox{iff}& \forall \goa \;e\trel \phi[x/\goa]\\
e\trel \exists x\phi\phantom{Ai}  &\mbox{iff}& \exists
\goa\;e\trel \phi[x/\goa]
 \end{eqnarray*}
\end{definition}

\begin{lemma}\label{Q1.3} If $e\trel\phi$ then
$\phi^{\circ}$.\end{lemma}
 \begin{proof}
 See Lemma 5.7 in \cite{R05}.
\end{proof}

\begin{lemma}\label{negself} Negated formulas are self-realizing, that is
to  say, if $\psi$ is a statement with parameters in $\vns_{\mathcal{A}}$, then
$$\neg \psi^{\circ}\,\to\, \mathbf{0}\trel \neg\psi.$$ \end{lemma}
\begin{proof}
 Assume $\neg \psi^{\circ}$. From $f\trel \psi$ we would get
$\psi^{\circ}$ by Lemma \ref{Q1.3}. But this is absurd. Hence
$\forall f\,\neg f\trel \psi$, and therefore $\mathbf{0}\trel \neg\psi$.
\end{proof}

\begin{lemma}\label{id} 
There are  closed application terms
${\mathbf{i_r,i_s,i_t,i_0,i_1}}$ such that for all {\em
$\goa,\gob,\goc\in \vns$,}
\begin{enumerate}
\item $ \mathbf{i_r}\trel \goa=\goa$.
 \item $ \mathbf{i_s}\trel\goa=\gob\rightarrow \gob=\goa$.
 \item $ \mathbf{i_t}\trel(\goa=\gob\wedge \gob=\goc)\,\rightarrow\, \goa=\goc$.
 \item $\mathbf{i_0}\trel (\goa=\gob\wedge \gob\in \goc)\,\rightarrow\,
\goa\in \goc$.
 \item $\mathbf{i_1}\trel (\goa=\gob\wedge \goc\in\goa)\,\rightarrow\, \goc\in \gob$.
 \item Moreover, for each
formula $\varphi(v,u_1,\ldots,u_k)$ of $\CZF$ all of whose free
variables are among $v, u_1,\ldots,u_k$ there exists a closed
application term ${\mathbf{i}}_{\varphi}$ such that for all
$\goa,\gob,\goc_1,\ldots,\goc_k\in \vns_{\mathcal{A}}$,
 $${\mathbf{i}}_{\varphi}\trel \varphi(\goa,\vec{\goc})\,\wedge\, \goa=\gob \;\rightarrow\;
\varphi(\gob,\vec{\goc}),$$ where
$\vec{\goc}=\goc_1,\ldots,\goc_k$.
\end{enumerate}
\end{lemma}
\begin{proof}
See \cite[Lemma 5.12]{R05}. 
\end{proof}
\begin{definition}{\em The {\em extended bounded formulas} are the smallest class
of formulas containing the formulas of the form $x\in y$, $x=y$,
$e\trel x\in y$, $e\trel x=y$ (where $x,y$ are variables or
elements of $\vns_{\mathcal{A}}$) which is closed under
$\wedge,\vee,\neg,\rightarrow$ and bounded
quantification.}\end{definition}
\begin{lemma}\label{sepex} $(\CZF)$ Separation holds for extended bounded formulas, i.e., for
every extended bounded formula $\varphi(v)$ and set $x$, $\{v\in
x\,:\;\varphi(v)\}$ is a set.
\end{lemma}
 \begin{proof}
  See \cite[Lemma 5.15]{R05}.
 \end{proof}
\begin{proposition}[Soundness theorem]\label{soundness}
 Let $S$ any combination of
 the axioms and schemes
$\mathbf{Full\;Separation}$, $\mathbf{Powerset}$, $\mathbf{REA}$, 
 $\MP$, $\mathbf{AC}_{\omega}$, $\mathbf{DC}$, $\mathbf{RDC}$, and
 $\mathbf{PAx}$.
 Then, for every theorem $\theta$ of $\CZF+S$, there exists an application term $t$ such that
 $\CZF+S\vdash (t\trel \theta)$. In particular, $\CZF$, $\CZF+\REA$, $\IZF$, $\IZF+\REA$
 satisfy this property.
 Moreover, the proof of this soundness theorem is effective in that the application term
 $t$ can be constructed from the $\CZF+S$ proof of $\theta$.
\end{proposition}
\begin{proof}
This is proved in the same way as \cite[Theorem 6.1, Theorem 7.2 and Theorem 9.1]{R05} and
 \cite[Theorem 7.4]{R08}.
\end{proof}
 \section{Realizing $\IPR$ for finite types}

In this section, we prove $\CZF$ and several constructive set theories are closed under
$\IPR$ for finite types. We fix an applicative structure $\mathcal{A}$
  such that
$|\mathcal{A}|$ and its graph $\{ \langle a, b, c\rangle\in|\mathcal{A}|^3:ab=c\}$
of its application are sets, as $\RE(\omega)$ in Section \ref{graphmodel}.
In what follows, we use $a, b,\ldots ,i$ for elements of $\mathcal{A}$.
By $\SIA$, $\N=\{ x\in|\mathcal{A}|:N(x)\}$ is a set.

We need several properties of ordered pairs in $\vns_{\mathcal{A}}$.
For for any $\goa$ and $\gob$, define 
$\underline{\{ \goa, \gob\}}$ by
\begin{align*}
 \langle\{ \goa^{\circ},
\gob^{\circ}\}, \{ \langle \bp\zero g, \goa\rangle:g\trel
\goa=\goa\}\cup\{ \langle \bp \one g, \gob\rangle:g\trel \gob=\gob\}\rangle.
\end{align*}
By Lemma \ref{sepex},
we can prove that $\second(\underline{\{\goa,\gob \}})$ is a
set such that for each $x\in \second\underline{\{ \goa, \gob\}}$,
$\first(\second(x))\in\{ \goa^{\circ}, \gob^{\circ}\}$.
By Lemma \ref{stets} (iii), we have $\underline{\{
\goa,\gob\}}\in\vns_{\mathcal{A}}$.  
The following lemma shows that $\underline{\{ \goa, \gob\}}$ acts as the pair of $\goa$
and $\gob$ in $\vns_{\mathcal{A}}$.
\begin{lemma}[$\CZF$]\label{Pair}
There is a closed term $t$ such that
$t\trel \forall x(x\in\underline{\{ \goa,\gob\}}\leftrightarrow
x=\goa\lor x=\gob)$. 
\end{lemma}
  \begin{proof}
   See ({\bf Pair}) and ({\bf Bounded Separation}) in the proof of \cite[Theorem 6.1]{R05}.
  \end{proof}

  We often write $\underline{\{ \goa\}}$ for
  $\underline{\{\goa,\goa\}}$. 
For ordered pair, we write $\underline{\langle \goa,\gob\rangle}$ for
$\underline{\{\underline{\{ \goa\}},
\underline{\{\goa,\gob\}}\}}$.


\begin{lemma}[$\CZF$]\label{op}
 There are closed terms $\t_{\rm op}$ and $\t_{\rm op'}$ such that
  \begin{align*}
   & \forall \god,
   \god', \goe, \goe'(\t_{\rm op}\trel
   \god=\mathfrak{d'}\land \goe=\goe'\to
   \underline{\langle\god,\goe\rangle}=
   \underline{\langle\god',\goe'\rangle}), \\
   & \forall \god, \god', \goe, \goe'
   (\t_{\rm op'}\trel
   \underline{\langle\god,\goe\rangle}=
   \underline{\langle\god',\goe'\rangle}\to
   \god=\mathfrak{d'}\land \goe=\goe'
   ).
  \end{align*}
\end{lemma}
\begin{proof}
This is implied by Proposition \ref{soundness}. 
\end{proof}

 Let $\Psi(\goa)$ be as follows: 
\begin{multline*}
 \forall a \forall \goe\forall \god(
 \langle a, \god\rangle\in \goa^*\land 
 \langle a, \goe\rangle\in \goa^*\to
 \god=\goe)\land \\
\exists b\forall \langle h, \goh\rangle, \langle h',
 \goh'\rangle\in\goa^*(\exists c(c\trel
 \goh=\goh')\to bhh'\trel\goh=\goh').
\end{multline*}
An intuitive idea for $\Psi(\goa)$ is that each element of $\goa^*$ is
injectively indexed by some element of $|\mathcal{A}|$ and $\goa$ has a canonical
realizer for the equality between its elements.
   \begin{lemma}[$\CZF$]\label{mainlemma}
  For each $\goa$ and $\gob$, there is $\goc$ such that
 \begin{align*}
  \forall \goa\forall\gob[
  \Psi(\goa)\land
  \Psi(\gob)\to 
  \exists\goc(\Psi(\goc)\land
  \exists a(a\trel \forall x(x\in \goc\leftrightarrow x\in
  \Func(\goa,\gob))))]. 
 \end{align*}
\end{lemma}
   \begin{proof}
    Assume $\Psi(\goa)\land \Psi(\gob)$.
    Take $\i_{\goa}$ and $\i_{\gob}$ such that
\begin{align*}
 & \forall \langle h, \goh\rangle, \langle h', \goh'\rangle\in \goa^*
 (\exists a(a\trel \goh=\goh')\to\i_{\goa}
 hh'\trel\goh=\goh'), \\
 & \forall \langle h, \goh\rangle, \langle h', \goh'\rangle\in \gob^*
 (\exists a(a\trel \goh=\goh')\to\i_{\gob}
 hh'\trel\goh=\goh'). 
\end{align*}
    For each $f\in|\mathcal{A}|$, define $\hat{f}$ and
    $f\in (\goa\Rightarrow\gob)$ as follows:
 \begin{align*}
 & \hat{f}=\{ \langle d, \underline{\langle \god,
   \goe\rangle}\rangle:\langle d,\god\rangle\in \goa^*\land 
  \exists e(fd=e\land \langle e, \goe\rangle\in \gob^*)\}, \\
  & f{\in}(\goa\Rightarrow\gob)\equiv
    \forall d, d'\forall \god,\god'(
  \langle d,\god\rangle\in\goa^*\land
  \langle d', \god'\rangle\in\goa^*\land
  \i_{\goa}dd'\trel \god=\god'\to \\
  & \hspace{30mm}\exists e,e'\exists \goe,\goe'(fd=e\land fd'=e'\land
  \langle e, \goe\rangle\in \gob^*\land
  \langle e',\goe'\rangle\in\gob^*\land
  \i_{\gob}ee'\trel \goe=\mathfrak{e'}), \\
  & f^{\it real}=\{ \langle
  \god^{\circ},\goe^{\circ}\rangle:
  \exists d(\langle d,\underline{\langle\god,\goe\rangle}\rangle\in \hat{f})\}.
 \end{align*}
    Then $\hat{f}$ is a set by Bounded Separation and $f\in(\goa\Rightarrow
    \gob)$ is equivalent to an extended bounded formula.
    Define $\goc$ by
\begin{align*}
 & \goc=
  \langle (\gob^{\circ})^{\goa^{\circ}}, \{\langle f,\langle f^{\it real},
  \hat{f}\rangle\rangle: f\in(\goa\Rightarrow\gob)\}\rangle.
 \end{align*}
    Then $\second(\goc)$ is a set by Strong Collection.
    For each $x\in\second(\goc)$, it has the form $\first(\second(x))=f^{\it real}$
    for some $f$ such that $f\in(\goa\Rightarrow\gob)$.
    Since $\gob$ satisfies
    $\forall a\forall \goe\forall\goe'(\langle a,\goe\rangle\in \gob^*\land\langle
    a,\goe'\rangle\in\gob^*\to \goe=\goe')$,
    we have $\first(\second(x))\in(\gob^{\circ})^{\goa^{\circ}}$. 
    Hence $\goc\in\vns_{\mathcal{A}}$.
    Then, it is easy to see that
\begin{align*}
 & \forall x(x\in\goc^{\circ}\leftrightarrow x\in x\in\Func(\goa^{\circ},\gob^{\circ})) &
 &\text{and} &
 & \forall a \forall \god\forall \god'(
 \langle a, \god\rangle\in \goc^*\land
 \langle a, \god'\rangle\in \goc^*\to
 \god=\god').
\end{align*}
    Define ${\bf i}_{\goc}$ as follows:
    \begin{align*}
       & \i_{\goc}=\lambda fg.\bp
  (\lambda h. \bp h(\t_{\rm op}(\bp(\i_{\goa}hh)(\i_{\gob}(fh)(gh)))))
  (\lambda h. \bp h(\t_{\rm op}(\bp(\i_{\goa}hh)(\i_{\gob}(gh)(fh))))).
 \end{align*}
   We show that 
    $\forall f, g\forall \gof,\gog
    (\langle f, \gof\rangle\in \goc^*\land
    \langle g,\gog\rangle\in\goc^*\land
    \exists b(b\trel\gof=\gog)\to
    \i_{\goc}fg\trel \gof=\gog)$.
    Assume $\langle f,\gof\rangle\in \goc^*$,
    $\langle g,\gog\rangle\in\goc^*$ and
    $b\trel \gof=\gog$.
    Then
    \begin{align*}
     & \gof^{\circ}=\gog^{\circ}\land
     \forall \langle
     h,\underline{\langle \goh,\goi\rangle}\rangle\in\gof^*
     \exists \underline{\langle \mathfrak{h'},
     \goi'\rangle}(\langle ((b)_0h)_0, \underline{\langle \mathfrak{h'},
     \goi'\rangle}\rangle\in\gog^*\land
     ((b)_0h)_1\trel
     \underline{\langle \goh,\goi\rangle}=
     \underline{\langle \goh',\goi'\rangle}).
    \end{align*}
    Fix $\langle h, \underline{\langle \goh,\goi\rangle}\rangle\in \gof^*$
    and take $\underline{\langle \goh',\goi'\rangle}\in \gog^*$
    such that $\langle ((b)_0h)_0, \underline{\langle \mathfrak{h'},
     \goi'\rangle}\rangle\in\gog^*\land
     ((b)_0h)_1\trel
     \underline{\langle \goh,\goi\rangle}=
     \underline{\langle \goh',\goi'\rangle}$.
    Then there are application terms $p$ and $q$ such that
    $p\trel\goh=\goh'$ and
    $q\trel\goi=\goi'$ by Lemma \ref{op}.
    Since $\langle h,\underline{\langle \goh,\goi\rangle}\rangle\in
    \gof^*$ implies $\langle h, \goh\rangle\in \goa^*$ and
    $\langle fh, \goi\rangle\in \gob^*$ and since
    $\langle ((b)_0h)_0,\underline{\langle \goh',\goi'\rangle}\rangle\in
    \gog^*$ implies
    $\langle ((b)_0h)_0, \goh'\rangle\in \goa^*$ and
    $\langle g((b)_0h)_0, \goi'\rangle\in\gob^*$,
    we have $\i_{\goa}h((b)_0h)_1\trel \goh=\goh'$ and
    $\i_{\gob}(fh)(g((b)_0h)_0)\trel \goi=\goi'$.
    By $g\in(\goa\Rightarrow\gob)$ and
    $\langle h, \goh\rangle\in\goa^*$,
    there is $\goi''$
    such that $\langle gh, \goi''\rangle\in \gob^*$,
    ${\bf i}_{\gob}(gh)(g((b)_0h)_0)\trel \goi''=\goi'$.
    Then we can construct $r$ such that $r\trel \goi=\goi''$
    by using ${\bf i}_s$ and ${\bf i}_t$ and so
    ${\bf i}_{\gob}(fh)(gh)\trel \goi=\goi''$.
    Then we have
    \begin{align*}
     \exists \goi''(\langle h,\underline{\langle
    \goh,\goi''\rangle}\rangle\in \gog^*\land
    \t_{\rm op}(\bp(\i_{\goa}hh)(\i_{\gob}(fh)(gh)))\trel
    \underline{\langle\goh,\goi\rangle}=\underline{\langle
    \goh,\goi'' \rangle}).
    \end{align*}
    In a similar way, we can show that, for each $\langle h,\underline{\langle
    \goh,\goi\rangle}\rangle\in \gog^*$,
    \begin{align*}
     \exists \goi''(\langle h,\underline{\langle
    \goh,\goi''\rangle}\rangle\in \gof^*\land
    \t_{\rm op}(\bp(\i_{\goa}hh)(\i_{\gob}(gh)(fh)))\trel
    \underline{\langle\goh,\goi\rangle}=\underline{\langle
    \goh,\goi'' \rangle}).
    \end{align*}
    Therefore, $\goi_{\goc}$ defined as above gives
    $\goi_{\goc}fg\trel \gof=\gog$.

    \medskip
    
    To show $\exists a(a\trel \forall x(x\in\goc\leftrightarrow
    x\in \Func(\goa,\gob)))$, 
    we have to construct $s$ and $t$ such that, for any $\gof$
\begin{align*}
& s\trel \gof\in\goc\to
 \forall y\in \goa\exists !z\in \gob(\langle
 y,z\rangle\in \gof \land
  \forall w\in \gof\exists y\in\goa\exists z\in\gob(w=\langle y,
 z\rangle)), \text{and} \\ 
 & t\trel
  \forall y\in \goa\exists !z\in \gob(\langle
 y,z\rangle\in \gof \land
  \forall w\in \gof\exists y\in\goa\exists z\in\gob(w=\langle y,
 z\rangle)) \to \gof\in\goc.
\end{align*}

First we construct $s$ with the above property.
 Assume that $a\trel \gof\in \goc$.
 Then,
 \begin{align}
  & a\trel \gof\in \goc \nonumber\\
  \leftrightarrow&
  \gof^{\circ}\in \Func(\goa,\gob)\land
  \exists\gog
  (\langle (a)_0, \gog\rangle\in\goc^*\land
  (a)_1\trel \gof=\gog) \nonumber \\ 
  \leftrightarrow &
  \gof^{\circ}\in \Func(\goa,\gob)\land \exists\gog
  (\langle (a)_0, \gog\rangle\in\goc^*\land
  \gof^{\circ}=\gog^{\circ} \nonumber \\
  & \qquad \land
  \forall h\forall \goh(\langle h,\goh\rangle\in \gof^*\to
  (a)_{10}h \trel \goh\in \gof) \nonumber \\
  & \qquad \land
  \forall h\forall \goh(\langle h,\goh\rangle\in
  \gog^*\to 
  (a)_{11}h \trel \goh\in \gog)) \nonumber \\
  \leftrightarrow &
  \gof^{\circ}\in \Func(\goa,\gob)\land \exists\gog
  (\langle (a)_0, \gog\rangle\in\goc^*\land
  \gof^{\circ}=\gog^{\circ}\nonumber \\
  &\qquad \land\forall h\forall \goh(\langle h,\goh\rangle\in \gof^*\to
  \goh^{\circ}\in\gog^{\circ}\land
  \exists\goi(\langle ((a)_{10}h)_0,\goi\rangle\in \gog^*\land
  ((a)_{10}h)_1\trel \goh= \goi)) \label{**}\\ 
  &\qquad \land\forall h\forall \goh(\langle h,\goh\rangle\in
  \gog^*\to \goh^{\circ}\in \gof^{\circ}\land 
  \exists \goi(\langle((a)_{11}h)_0,\goi\rangle\in \gof^*\land
  ((a)_{11}h)_1 \trel \goh=\goi)))\label{****},
 \end{align}
and so take $\gog$ such that
\begin{align}
 \langle (a)_0,\gog\rangle\in \goc^*\land (a)_1\trel
    \gof=\gog. \label{***}
\end{align}

 Since $\forall d\forall \god(
  \langle d,\god\rangle\in \goa^* \to 
  \exists e\exists \goe (ad=e\land
    \langle e,\goe\rangle\in\gob^*\land
    \langle d, \underline{\langle
    \god,\goe\rangle}\rangle\in\gog^*))$, we have, by
    (\ref{****}),
    \begin{align*}
  \forall d\forall \god(
  \langle d,\god\rangle\in \goa^* \to 
  \exists e\exists \goe (ad=e\land
  \langle e,\goe\rangle\in\gob^*\land
  \exists \goi(\langle ((a)_{11}d)_0,\goi\rangle\in \gof^*\land
  ((a)_{11}d)_1\trel \underline{\langle \god,\goe\rangle}=\goi))).
    \end{align*} 
Therefore $s'\equiv \lambda d.(a)_{11}d$ satisfies, for $a$ such that
    $a\trel\gof\in\goc$, 
   \begin{align*}
   s'a\trel \forall y\in\goa\exists z\in\gob(\langle y,z\rangle\in\gof).
  \end{align*}
    
    Assume that $\langle d,\god\rangle\in\goa^*$,
    $\langle e,\goe\rangle,\langle e',\goe'\rangle\in\gob^*$ and
    $b\trel \underline{\langle \god,\goe\rangle}\in \gof\land
    \underline{\langle \god,\goe'\rangle}\in\gof$ and 
    take $\goh$ and $\goi$ such that
 \begin{align*}
  & \langle (b)_{00}, \goh\rangle\in \gof^*\land
  (b)_{01}\trel \underline{\langle \god,\goe\rangle}=\goh, &
  & \langle (b)_{10}, \goi\rangle\in \gof^*\land
  (b)_{11}\trel \underline{\langle \god,\goe'\rangle}=\goi. 
 \end{align*}
    Then, by (\ref{**}) and (\ref{***}),
    we have $\goj$, $\goj'$, $\mathfrak{k}$ and
    $\mathfrak{k}'$ such that  
    \begin{align*}
  &\goh^{\circ}\in\gog^{\circ}\land
  \langle ((a)_{10}(b)_{00})_0,\underline{\langle \goj,\mathfrak{k}\rangle }\rangle\in \gog^*\land
     ((a)_{10}(b)_{00})_1\trel \goh= \underline{\langle \goj,\mathfrak{k}\rangle},\\
    &   \goi^{\circ}\in\gog^{\circ}\land
  \langle ((a)_{10}(b)_{10})_0,\underline{\langle \goj', \mathfrak{k}'\rangle} \rangle\in \gog^*\land
  ((a)_{10}(b)_{10})_1\trel \goi= \underline{\langle \goj', \mathfrak{k}'\rangle}.
    \end{align*}
    Define $p_0$, $p_1$, $p_2$, $p_3$, $p_4$, $p_5$, $p_6$, $p_7$ and $p_8$ be as follows:
    \begin{align*}
     & p_0\equiv \lambda ab.(a)_{10}(b)_{00}, &
     & p_1\equiv \lambda ab.(a)_{10}(b)_{10}, & \\
     & p_2\equiv \lambda ab. \mathbf{i_t}(\bp(b)_{01}(p_0ab)_1), &
     & p_3\equiv \lambda ab. \mathbf{i_t}(\bp(b)_{11}(p_1ab)_1),&\\
     &
     p_4\equiv \lambda ab.
     \mathbf{i_t}(\bp(\i_s(\bp_0(\t_{\rm op'}(p_2ab))))(\bpn(\t_{\rm op'}(p_3ab)))) &
     &
     p_5\equiv
     \lambda ab.\i_{\gob}(g(p_0ab)_0)(g(p_1ab)_0)))
     \\
     & p_6\equiv \lambda ab.\bpo(\t_{\rm op'}(p_2ab)) &
     & p_7\equiv \lambda ab. \bpo(\t_{\rm op'}(p_3ab)) \\
     &\multispan3{$p_8\equiv
     \lambda ab.\mathbf{i_t}(\bp(p_7ab)(\i_t(\bp(p_6ab)(p_5ab))).$\hfill}
    \end{align*}
Then we have
     \begin{align*}
      & p_2ab\trel \underline{\langle\god,\goe\rangle}=
      \underline{\langle\goj,\mathfrak{k}\rangle}, &
      & p_3ab\trel \underline{\langle\god,\goe'\rangle}=
      \underline{\langle\goj',\mathfrak{k}'\rangle}, & \\
     & \bpn(\t_{\rm op'}(p_2ab))\trel \god=\goj, &
     & \bpn(\t_{\rm op'}(p_3ab))\trel \god=\goj', &
      & \i_s(\bpn(\t_{\rm op'}(p_2ab)))\trel \goj=\god, & \\
      & p_4ab\trel
      \goj=\goj', &
      & \i_{\goa}(p_0ab)_0(p_1ab)_0\trel \goj=\goj', & 
      & p_5ab\trel\mathfrak{k}=\mathfrak{k}', & \\ 
      & p_6ab\trel \goe=\mathfrak{k}, &
      & p_7ab\trel \goe'=\mathfrak{k}'&
      & p_8ab\trel \goe=\mathfrak{k}'. &
     \end{align*}
    Therefore $s''\equiv \lambda ab.
    \mathbf{i_t}(\bp(p_8ab)(\mathbf{i_s}(p_7ab)))$ satisfies
    $s''ab\trel \goe=\mathfrak{e'}$, when $a\trel
    \gof\in\goc$ and
    $b\trel \underline{\langle \god,\goe\rangle}\in \gof\land
    \underline{\langle \god,\goe'\rangle}\in\gof$.

    \smallskip

    Again by (\ref{**}), we have
    \begin{align*}
     \forall h\forall \goh(\langle h,\goh\rangle\in \gof^*\to
  \goh^{\circ}\in\gog^{\circ}\land
  \exists\goi(\langle ((a)_{10}h)_0,\goi\rangle\in \gog^*\land
  ((a)_{10}h)_1\trel \goh= \goi)),
    \end{align*}
    which implies
    \begin{multline*}
     \forall h\forall \goh(\langle h,\goh\rangle\in \gof^*\to
  \goh^{\circ}\in\gog^{\circ}\land
     \exists\god\exists \goe(
     \langle ((a)_{10}h)_0,\god\rangle\in \goa^*\land
     \langle g((a)_{10}h)_0,\goe\rangle\in \gob^*\land \\
     \langle ((a)_{10}h)_0,
     \underline{\langle \god,\goe\rangle}\rangle\in
     \gog^*\land 
     ((a)_{10}h)_{11}\trel \goh=
     \underline{\langle \god,\goe\rangle})).
    \end{multline*}
    Hence $s'''\equiv\lambda ah.\bp((a)_{10}h)_0(\bp(g((a)_{10}h)_0)((a)_{10}h)_{11})$ satisfies
    $s'''a\trel\forall w\in\gof\exists
    y\in \goa\exists z\in\gob
    (w=\langle y,z\rangle)$,
    when $a\trel \gof\in\goc$.
 
 \smallskip
 
 Therefore $s\equiv \lambda a.\bp(\bp(s'a)(\lambda b.s''ab))(s'''a)$ satisfies
 \begin{align*}
s\trel\gof\in\goc\to
 \forall y\in \goa\exists !z\in \gob(\langle
 y,z\rangle\in \gof) \land
  \forall w\in \gof\exists y\in\goa\exists z\in\gob
  (w=\langle y, z\rangle).
 \end{align*}

 \bigskip

 To construct $t$ such that
 \begin{align*}
  t\trel   \forall y\in \goa\exists !z\in \gob(\underline{\langle
 y,z\rangle}\in \gof \land
  \forall w\in \gof\exists y\in\goa\exists z\in\gob(w=\underline{\langle y,
 z\rangle})) \to \gof\in\goc,
 \end{align*}
    assume that $a\trel
    \forall y\in \goa\exists !z\in \gob(\underline{\langle
 y,z\rangle}\in \gof \land
  \forall w\in \gof\exists y\in\goa\exists z\in\gob(w=\underline{\langle y,
 z\rangle})).$
 Then, we have
 \begin{align}
  & (a)_{00}\trel\forall y\in\goa\exists
  z\in\gob(\underline{\langle y,z\rangle} \in\gof), \label{4}\\
  & (a)_{01}\trel\forall y\in\goa,\forall z,w\in\gob
  (\underline{\langle y,z\rangle}\in \gof\land
  \underline{\langle y,w\rangle}\in\gof\to
  y=w), \label{5}\\
  & (a)_1\trel \forall w\in \gof\exists y\in\goa\exists z\in\gob
  (w=\underline{\langle y,z\rangle}). \label{6}
 \end{align}
 (\ref{4}) implies
 \begin{align*}
  \forall d\forall \god(
  \langle d,\god\rangle\in \goa^*\to\exists
  \goe(\langle ((a)_{00}d)_0,\goe\rangle\in \gob^*
  \land ((a)_{00}d)_1\trel \underline{\langle
  \god,\goe\rangle}\in\gof)). \tag{\ref{4}$'$}\label{42}
 \end{align*}
 Set $g\equiv\lambda d. ((a)_{00}d)_0$.
    If $\langle d, \god\rangle\in \goa^*$,
    $\langle d', \goa^*\rangle\in \goa^*$ and
    ${\bf i}_{\goa}dd'\trel \god=\god'$,
    then there are $\goe$ and $\goe'$ such that
\begin{align*}
&     \langle gd, \goe\rangle\in\gob^* \land
 ((a)_{00}d)_1\trel \underline{\langle \god,\goe\rangle}
 \in\gof &
 & \text{and} &
&     \langle gd', \goe'\rangle\in\gob^* \land
 ((a)_{00}d')_1\trel \underline{\langle \god',\goe'\rangle}
 \in\gof.
\end{align*}
    By (\ref{5}) and ${\bf i}_{\goa}dd'\trel \god=\god'$, we have
    $b$ such that $b\trel \goe=\goe'$ and so
    ${\bf i}_{\gob}(gd)(gd')\trel\goe=\goe'$.
    Hence $g\in(\goa\Rightarrow\gob)$. 

    First we show 
    $\langle g,\langle \gof^{\circ},\hat{g}\rangle\rangle\in\goc^*$.
    It is enough to show $g^{\it real}=\{ \langle \god^{\circ},\goe^{\circ}\rangle:\exists d(\langle
    d,\underline{\langle\god,\goe\rangle}\rangle\in \hat{g})\}=\gof^{\circ}$.
    Take any $x\in \hat{g}$.
    Then $x=\langle d, \langle\god,\goe\rangle\rangle$ 
    for some $\langle d,\god\rangle\in \goa^*$ and
    $\langle e,\goe\rangle\in \gob^*$ such that $e=gd=((a)_{00}d)_{0}$.
    By (\ref{4}), we have $\goe'$ such that $\langle
    e,\goe'\rangle\in\gob^*\land
    ((a)_{00}d)_1\trel \underline{\langle \god,\goe'\rangle}\in
    \gof$.
    Since $\Psi(\gob)$ yields
    $\langle e, \goe\rangle\in \gob^* \land\langle
    e,\goe'\rangle\in\gob^*\to \goe=\goe'$,
    we have $\goe=\goe'$ and so
    $\first(\second(x))=\langle \god^{\circ},\goe^{\circ}\rangle$,
    which is in $\gof^{\circ}$.
    Conversly, for each $w\in\gof^{\circ}$, 
    there are $d$, $\god$ and $\goe$ such that $\langle d,\god\rangle\in \goa^*$,
    $\langle gd,\goe\rangle\in\gob^*$ and $w=\langle \god^{\circ},\goe^{\circ}\rangle$
    by (\ref{42}) and (\ref{6}) and so
  $\langle d,\underline{\langle\god,\goe\rangle}\rangle\in \hat{g}$.
 
    Set $\gog=\langle \gof^{\circ},\hat{g}\rangle$.
We show that
 there is an application term $t'$ such that $t'\trel
 \gof=\gog$.
 Assume $\langle h, \goh\rangle\in \gof^*$.
 By (\ref{6}), there are $\god$ and
    $\goe$ such that
    $\langle ((a)_1h)_0, \god\rangle\in\goa^*$, $\langle ((a)_{1}h)_{10},
    \goe\rangle\in\gob^*$ and 
    $((a)_1h)_{11}\trel \goh=\underline{\langle
    \god,\goe\rangle}$, which implies 
    $\bp h(\mathbf{i_s}((a)_1h)_{11})\trel\underline{\langle \god,\goe\rangle}\in \gof$.
By (\ref{4}), there is $\goe'$ such that
    $\langle g((a)_1h)_0,\goe'\rangle\in\gob^*$ and
    $((a)_{00}((a)_1h)_0)_1\trel
    \underline{\langle\god,\goe'\rangle}\in\gof$.
    Therefore we have $q_0agh\trel \goe=\goe'$, where
    \begin{align*}
     q_0\equiv\lambda a g h.(a)_{01}((a)_1h)_0((a)_1h)_{10}(g((a)_1h)_0)(\bp (\bp
    h(\mathbf{i_s}((a)_1h)_{11}))((a)_{00}((a)_1h)_0)_1),
    \end{align*}
    which implies
    \begin{align*}
     & q_1agh\trel
    \underline{\langle\god,\goe\rangle}=\underline{\langle \god,\goe'\rangle}&
     & q_2agh\trel
     \goh=\underline{\langle \god,\goe'\rangle},\;\text{where} \\
     & q_1\equiv\lambda agh.\t_{\rm op}(\bp \mathbf{i_r}(q_0agh)) &
     & q_2\equiv\lambda agh.
     \mathbf{i_t}(\bp ((a)_1h)_{11}(\t_{\rm op}(\bp
     \mathbf{i_r}(q_0agh)))).
    \end{align*}
    By the construction of $\hat{g}$, $\langle ((a)_1h)_0,
    \underline{\langle \god,\goe'\rangle}\rangle\in\gog^*$, 
    and so we have
    \begin{align*}
     \forall h\forall \goh(\langle h,\goh\rangle\in
\gof^{*}\to \bp ((a)_1h)_0(q_2agh)\trel \goh\in\gog).
    \end{align*}
Hence we have $\forall h\forall \goh(\langle h,\goh\rangle\in
\gof^{*}\to \exists \goi(\langle ((a)_1h)_0, \goi\rangle\in
 \gog\land ((a)_1h)_1\trel \goh=\goi))$.
 Next, assume that $\langle h, \goh\rangle\in \gog^*$.
 Since $\gog^*=\{ \langle d,\underline{\langle \god,
    \goe\rangle}\rangle:\langle d,\god\rangle\in\goa^*\land 
    \exists e(gd=e\land\langle e,\goe\rangle\in \gob^*\}$, we have
  $\goh=\underline{\langle \god,\goe\rangle}$ for some $\langle
    h,\god\rangle\in \goa^*$ and
 $\langle e,\goe\rangle\in\gob^*$ such that $gh=e$.
    Recall that $g=\lambda d.((a)_{00}d)_0$ and (\ref{42}).
    Take $\goe'$ such that
    $\langle ((a)_{00}h)_0,\goe'\rangle\in \gob^*
  \land ((a)_{00}h)_1\trel \underline{\langle
  \god,\goe'\rangle}\in\gof)$.
    Again by $\Psi(\gob)$, we have $\goe=\goe'$. 
 Then $\exists \goi(\langle ((a)_{00}h)_{10}, \goi\rangle\in
 \gof^* \land ((a)_{00}h)_{11}\trel \underline{\langle
 \god,\goe\rangle}= \goi)$.
 Hence $\forall h\forall \goh(\langle  h,\goh\rangle\in
 \gog^*\to \exists \goi(\langle ((a)_{00}h)_{10},
 \goi\rangle\in\gof^*\land
 ((a)_{00}h)_{11}\trel \goh= \goi)))$.
 Therefore $\bp(\lambda d.(a)_1d)(\lambda d.((a)_{00}d)_1)\trel
 \gof=\gog$.
    
 Set $t\equiv \lambda a.\bp (\lambda d.((a)_{00}d)_0)(\bp(\lambda d.(a)_1d)(\lambda d.((a)_{00}d)_1))$. Then
 \begin{align*}
  t \trel\forall y\in \goa\exists !z\in \gob(\underline{\langle
 y,z\rangle}\in \gof \land
  \forall w\in \gof\exists y\in\goa\exists z\in\gob(w=\underline{\langle y,
 z\rangle})) \to \gof\in\goc.
 \end{align*}
   \end{proof}

 \begin{lemma}\label{membership}
  For any $\goa$, $\gob$ in $\vns_{\mathcal{A}}$ and $a$ in $\mathcal{A}$,
  $\langle a,\goa\rangle\in\gob^*$ implies $\bp a{\bf i_r}\trel
  \goa\in\gob$.
 \end{lemma}
 \begin{proof}
  It follows from the definition of $e\trel \goa\in\gob$.
 \end{proof}

For each $n\in \omega$, define $n_{\mathcal{A}}$ by
\begin{align*}
& 0_{\mathcal{A}}\equiv \mathbf{0} &
 & (n+1)_{\mathcal{A}}\equiv\mathbf{s}_{N}n_{\mathcal{A}}, &
\end{align*}
and let
\begin{align*}
 & \underline{n}=\langle n, \{\langle m_{\mathcal{A}}, \underline{m}\rangle:m<n\}\rangle &
 & \underline{\omega}=\langle \omega,\{ \langle n_{\mathcal{A}},\underline{n}\rangle:n\in\omega\}\rangle.
\end{align*}
\begin{lemma}[$\CZF$]\label{omega}
  There is $a$ such that
 \begin{align*}
a\trel\forall y[y\in \underline{\omega}\leftrightarrow \forall z(z\in
	    y\leftrightarrow \bot)\lor (\exists w\in \underline{\omega})\forall v(v\in
  y\leftrightarrow v\in w\lor v=w)].  
 \end{align*}
\end{lemma}
\begin{proof}
 See \cite[Theorem 6.1 (Infinity)]{R05}.
\end{proof}
Recall that each finite type is coded by a natural number and we do not distinguish a type
$\sigma$ and its code.
By $\underline{\sigma}$, we mean $\underline{n}$ for the code $n$ of a type $\sigma$.
\begin{lemma}[$\CZF$]\label{type1}
 For each finite type $\sigma$, there is $\underline{\bo{\sigma}}$ such that
$\Psi(\underline{\bo{\sigma}})$ and
$\exists a(a\trel \phi(\underline{\sigma}, \underline{\bo{\sigma}}))$.
 \begin{proof}
  This is proved by induction on the type structure.
  For $\sigma=0$, we can prove $\Psi(\underline{\omega})$ as follows.
  It is clear that $\underline{\omega}^{\circ}=\omega$.
  For each $a$, $\goc$ and $\god$, if $\langle a, \goc\rangle\in
  (\underline{\omega})^*\land\langle a,\god\rangle\in(\underline{\omega})^*$, then there
  is $n\in\omega$ such that $a=n_{\mathcal{A}}$ and so
  $\goc=\god=\underline{n}$
  by the construction of $\underline{\omega}$, which implies $\goc=\god$.
  Assume that $b\trel\underline{n}=\underline{m}$ for some $n$, $m\in\omega$ and
  $b\in|\mathcal{A}|$.
  Then $(\underline{n})^{\circ}=(\underline{m})^{\circ}$ and so
  $\underline{n}=\underline{m}$.
  Therefore $c=\bp(\lambda x.x)(\lambda x.x)$ satisfies
  $c\trel \underline{n}=\underline{m}$.
  By Lemma \ref{type} and Lemma \ref{omega}, we have $a$ such that
  $a\trel\phi({\underline{0}}, \underline{\omega})$.
  
  Next we prove the induction step.
  Assume that $\sigma\equiv\arrowtype{\tau}{\rho}$.
 By the induction hypothesis, we have $\underline{\bo{\tau}}$, $\underline{\bo{\rho}}$,
  $p$ and $q$ such that
  \begin{align*}
   & \Psi(\underline{\bo{\tau}}, \bo{\tau})  &
   & p\trel \phi({\tau}, \underline{\bo{\tau}}); \\
   & \Psi(\underline{\bo{\rho}}, \bo{\rho})  &
   & q\trel \phi({\rho}, \underline{\bo{\rho}}).
  \end{align*}
  By Lemma \ref{mainlemma}, we can take $\underline{\arrowtype{\bo{\tau}}{\bo{\rho}}}$
  satisfying 
  $\Psi(\underline{\bo{\tau}^{\bo{\rho}}})$.
  Set $\underline{\bo{\sigma}}=\underline{{\bo{\tau}}^{\bo{\rho}}}$.
 Then there is $r$ such that $r\trel \forall
 w(w\in\underline{\bo{\sigma}}\leftrightarrow
 w\in\Func(\underline{\bo{\tau}},\underline{\bo{\rho}}))$.
 Since $\CZF$ proves that $\forall x\forall y
 (\phi(\tau, x)\land\phi(\tau, y)\to x=y)$ and
 $\forall x\forall y
 (\phi(\rho,x)\land\phi(\rho,y)\to x=y)$ by Lemma \ref{type} and
 that $\forall x\forall y\forall x'\forall y'(x=x'\land y=y'\to \forall
 w(w\in\Func(x,y)\leftrightarrow w\in\Func(x', y')))$,
 we can construct a realizer for   
 $\phi(\underline{\sigma}, \underline{\bo{\sigma}})$.
 \end{proof}

\end{lemma}
\begin{lemma}
   Let $\varphi(x)$ be a formula whose unique free variable is $x$.
   Then, for any finite type $\sigma$,
 $\CZF\vdash \varphi(\bo{\sigma})$ 
   implies $\CZF\vdash \exists a(a\trel \varphi(\underline{\bo{\sigma}}))$.
\end{lemma}
  \begin{proof}
   Assume $\varphi(x)$ is a formula whose unique free variable is $x$.
   If $\CZF\vdash\varphi(\bo{\sigma})$, then we have
   $\CZF\vdash\forall x(\phi(\sigma,x)\to \varphi(x))$.
   By Proposition \ref{soundness} and Lemma \ref{type1}, we have $t$ and $s$ such that
   $t\trel\forall x(\phi(\underline{\sigma},x)\to \varphi(x))$ and $s\trel
   \phi(\underline{\sigma}, \underline{\bo{\sigma}})$.
   Hence we have $ts\trel\varphi(\underline{\bo{\sigma}})$. 
  \end{proof}

  \begin{lemma}\label{0}
   For any finite type $\sigma$, $\CZF\vdash\exists x(x\in\bo{\sigma})$.
  \end{lemma}
  \begin{proof}
   This is proved by induction on types.
   For $\sigma=0$, $\emptyset\in\omega=\bo{\sigma}$.
   Assume that $0_{\tau}\in\bo{\tau}$.
   Then $\lambda x^{\sigma}.0_{\tau}\in\bo{\sigma}\to\bo{\tau}$.
  \end{proof}
  
  In the proof of 1 of the following theorem, the totality of the applicative structure is crucial.
  
  \begin{theorem}\label{theorem}
   Let $\T$ be a set theory such that Proposition \ref{soundness} holds and let
   $\varphi(\vec{y})$, $\psi(x,\vec{y})$ and $\theta(\vec{y})$ be formulae whose free variables are all
   displayed. 
   \begin{enumerate}
    \item\label{theorem1}   If $\T\vdash \forall \vec{y}(\neg \varphi(\vec{y})\to(\exists
	 x\in\bo{\sigma})\psi(x, \vec{y}))$, 
   then $\T\vdash \exists x\forall\vec{y}(\neg\varphi(\vec{y})\to
   x\in\bo{\sigma}\land\psi(x,\vec{y}))$.
    \item\label{theorem2} If $\T\vdash\forall \vec{y}[\forall
	 z\theta(\vec{y},z)\to \exists x\in\bo{\sigma}\psi(x,\vec{y})]$ and
	 $\T\vdash\forall \vec{y}\forall z(\theta(\vec{y},z)\lor\neg\theta(\vec{y},z))$,
	 then $\T\vdash \exists x\in\bo{\sigma}\forall \vec{y}[\forall
	 z\theta(\vec{y},z)\to\psi(x,\vec{y})]$. 
    \item\label{theorem3}
   If $\T\vdash \forall \vec{y}(\neg\varphi(\vec{y})\to\exists
   x\in\bo{\sigma}\psi(x,\vec{y}))$ and 
   $\T\vdash \exists \vec{y}\neg \varphi(\vec{y})$,
   then $\T\vdash \exists x\in\bo{\sigma}\forall
	     \vec{y}(\neg\varphi(\vec{y})\to\psi(x,\vec{y}))$.
   \end{enumerate}
  \end{theorem}
  \begin{proof}
   Assume that $\T$ is a set theory and $\varphi(\vec{y})$, $\psi(x,\vec{y})$ and
   $\theta(\vec{y})$ are formulas as they are stated.
Fix a total applicative structure $\mathcal{A}$ such that both $|\mathcal{A}|$ and the
   graph $\{\langle \langle a,b\rangle,c\rangle\in|\mathcal{A}|^{3}:\appo(a,b,c)\}$ are sets in $\T$, such as
   $\mathcal{RE}(\omega)$. 

   \ref{theorem1}.
   Assume $\T\vdash\forall \vec{y}(\neg\varphi(\vec{y})\to(\exists
   x\in\bo{\sigma})\psi(x,\vec{y}))$. 

   We reason in $\T$.
   By Lemma \ref{type1},
   we have an application term $t$ such that $t\trel\forall\vec{\goa}(
   \neg\varphi[\vec{y}/\vec{\goa}]\to
   (\exists x\in\underline{\bo{\sigma}}\psi[x,\vec{y}/\vec{\goa}]))$.
     Note that $(t0)_0{\downarrow}$ and $(t0)_1{\downarrow}$ because of the totality of
   $\mathcal{A}$.
   Take the following $c$.
\begin{align*}
 c=\langle \bigcup\{ \gob^{\circ}:
   \langle (t\mathbf{0})_0, \gob\rangle\in (\underline{\bo{\sigma}})^*\},
   \bigcup\{ \gob^*:\langle (t\mathbf{0})_0,
   \gob\rangle\in(\underline{\bo{\sigma}})^*\}\rangle.
\end{align*}
   
   Then, for each $x\in \second(c)$,
   there is $\gob$ such that
   $x\in \gob^*$ and $\langle (t\mathbf{0})_0, \gob\rangle\in
   (\underline{\bo{\sigma}})^*$, which implies 
   $\first(\second(x))\in \first(c)$.
   Hence $c\in\vns_{\mathcal{A}}$. Let $\goc=c$.
   By Lemma \ref{negself}, if $a\trel \neg\varphi[\vec{y}/\vec{\goa}]$,
   then $\mathbf{0}\trel\neg\varphi[\vec{y}/\goa]$ and so
   $t\mathbf{0}\trel (\exists x\in\underline{\bo{\sigma}})\psi[x,\vec{y}/\goa]$,
   which implies $\exists\gob(\langle
   (t\mathbf{0})_0,\gob\rangle\in(\underline{\bo{\sigma}})^*\land 
   (t\mathbf{0})_1\trel \psi[x/\gob,\vec{y}/\vec{\goa}])$.
   By the definition of $\underline{\bo{\sigma}}$,
   $\gob$ such that $\langle
   (t\mathbf{0})_0,\gob\rangle\in(\underline{\bo{\sigma}})^*$ is unique.
   Hence $\goc=\langle \bigcup \{ \gob^{\circ}\},
   \bigcup\{ \gob^*\}\rangle=\gob$ for such $\gob$.
   Therefore
   $\bp (t\mathbf{0})_0\mathbf{i_r}\trel\goc\in\underline{\bo{\sigma}}$
   by Lemma \ref{membership} and
   $(t\mathbf{0})_1\trel \psi[x/\goc,\vec{y}/\vec{\goa}]$.
   Let $s\equiv\lambda a.\bp(\bp (t\mathbf{0})_0\mathbf{i_r})(t\mathbf{0})_1$,
   where $a$ is a fresh variable.
   Then $\forall\vec{\goa}(s\trel \neg\varphi[\vec{y}/\vec{\goa}]\to
   \goc\in\underline{\bo{\sigma}}\land
   \psi[x/\goc, \vec{y}/\vec{\goa}])$ and so we
   have $s\trel\exists x\forall \vec{y}(\neg\varphi(\vec{y})\to
   x\in\underline{\bo{\sigma}}\land\psi(x,\vec{y}))$.
   By Lemma \ref{Q1.3}, we have $\exists x\forall \vec{y}(\neg\varphi(\vec{y})\to
   x\in\underline{\bo{\sigma}}\land \psi(x,\vec{y}))$.

   \ref{theorem2}.
   Assume that $\T\vdash\forall \vec{y}[\forall
	 z\theta(\vec{y},z)\to \exists x\in\bo{\sigma}\psi(x,\vec{y})]$ and
	 $\T\vdash\forall \vec{y}\forall z(\theta(\vec{y},z)\lor\neg\theta(\vec{y},z))$.

   We reason in $\T$.
   By Proposition \ref{soundness}, there are application terms $t$ and $s$ such that
   $t\trel \forall \vec{y}[\forall
	 z\theta(\vec{y},z)\to \exists x\in\underline{\bo{\sigma}}\psi(x,\vec{y})]$ and
   $s\trel\forall \vec{y}\forall z(\theta(\vec{y},z)\lor\neg\theta(\vec{y},z))$.
   Then $(s)_0=\mathbf{0}$ or $(s)_1=\mathbf{1}$.

   If $(s)_0=\mathbf{0}$, then we have $(s)_1\trel\forall z\theta(\vec{y}/\vec{\goa},z)$
   and so there is $\goc$ such that $\langle (t(s)_1)_0, \goc\rangle\in(\underline{\bo{\sigma}})^*\land
   (t(s)_1)_1\trel \psi(\goc,\vec{\goa})$
    for any $\vec{\goa}$.
   Therefore $\lambda a.(t(s)_1)_1\trel \forall \vec{y}(\forall
   z\varphi(\vec{y},z)\to\psi(\goc,\vec{y}))$ and so
   \begin{align*}
    \bp(t(s)_1)_0(\lambda a.(t(s)_1)_1)\trel
    \exists x\in\underline{\bo{\sigma}}
    (\forall z\varphi(\vec{y},z)\to\psi(\goc,\vec{y}))
   \end{align*}
   which implies $\exists x \in
    \bo{\sigma}(\forall \vec{y}(\forall z\varphi(\vec{y},z)\to \psi(\goc,\vec{y})))$.

   If $(s)_0=\mathbf{1}$,  then we have $(s)_1\trel\neg\theta(\vec{y}/\vec{\goa},z/\god)$
    for any $\vec{\goa}$ and $\god$.
   Hence, for any $\vec{\goa}$, there is no $a$ such that $a\trel \forall
   z\theta(\vec{y}/\vec{\goa}, z)$.
   By Lemma \ref{0}, there is $\langle b, \gob\rangle\in\underline{\bo{\sigma}}$.
   Then we have
   $\lambda a.0\trel \forall \vec{y}(\forall z\varphi(\vec{y},z)\to \psi(\gob,\vec{y}))$
   and so
   \begin{align*}
    \bp b(\lambda a.0)\trel\exists x\in\underline{\bo{\sigma}}
    \forall \vec{y}(\forall z\varphi(\vec{y},z)\to \psi(\gob,\vec{y})),
   \end{align*}
   which implies
   $\exists x\in\bo{\sigma}
   \forall \vec{y}(\forall z\varphi(\vec{y},z)\to \land\psi(\goc,\vec{y}))$.
   
   \ref{theorem3}.
    Assume that $\T\vdash \forall \vec{y}(\neg\varphi(\vec{y})\to\exists
   x\in\bo{\sigma}\psi(x,\vec{y}))$ and 
   $\T\vdash \exists \vec{y}\neg \varphi(\vec{y})$.

 We reason in $\T$.
By Proposition \ref{soundness}
   there are application terms $t$ and $s$ such that $t\trel\neg\varphi[\vec{y}/\vec{\goa}]\to 
   \exists x\in\underline{\bo{\sigma}}\psi[x, \vec{y}/\vec{\goa}]$
   for any $\vec{\goa}$ and
   that $s\trel\neg\varphi[\vec{y}/\vec{\god}]$.
  Hence we have
   $\mathbf{0}\trel \neg\varphi[\vec{y}/\vec{\god}]$ by Lemma \ref{negself}.
   This ensures that there is $\goc$ such that $\langle
   (t0)_0,\goc\rangle\in(\underline{\bo{\sigma}})^*$ and so
   $\exists \goc(\langle (t0)_0,\goc\rangle\in \underline{\bo{\sigma}}^*\land
   \lambda a.(t0)_1\trel\forall \vec{y}(\neg\varphi[\vec{y}/\vec{\goa}]\to \psi[x/\goc,
   \vec{y}/\vec{\goa}]))$.
   The last implies $\T\vdash \exists x\forall \vec{y}(\varphi(\vec{y})\to\psi(x,\vec{y}))$.
  \end{proof}

\section{Open problems}
We conclude the paper with some open problems.
\begin{problem}[Independence of Premise Rule in general]
Is the following independence of premise rule an admissible rule of $\CZF$ or any other familiar constructive/intuitionistic set
 theory $\T$?
 \begin{quote}
  If $\T\vdash \neg\varphi\to \exists x\psi(x)$,
  then $\T\vdash\exists x(\neg\varphi\to\psi(x))$,
  where $\varphi$ and $\psi$ have no free variables other than displayed, and 
  where $x$ is not free in $\varphi$.
 \end{quote}
 In Theorem \ref{theorem}.\ref{theorem1}, we proved it in the case in which $\exists x$ is
bounded by some finite type.
We do not know yet how whether this generalizes to other bounded for $\exists x$ or whether we can
 even remove it.
 The key to generalize this bound seems to be to construct a (total) PCA which injectively
 represents each element of the bound, like in Lemma \ref{mainlemma}.
\end{problem}

\begin{problem}[Choice Rule, $\ACR$]
 In \cite[3.7.5]{troelstra73}, it was shown that $\mathbf{HA}^{\omega}$ is closed under
 the choice rule for finite type (dubbed ACR there), i.e., 
 \begin{quote}
  If $\mathbf{HA}^{\omega}\vdash \forall x^{\sigma}\exists y^{\tau}\varphi(x,y)$, then
  $\mathbf{HA}^{\omega}\vdash\exists 
  z^{\arrowtype{\sigma}{\tau}}\forall x^{\sigma}\varphi(x,zx)$.
 \end{quote}
 We expect that the set theories we treat in this paper are also closed under $\ACR$, but
 we do not know yet if we can prove it with 
the model of PCA in this paper but problems with extensionality one faces seem to render
 it highly unlikely. 
\end{problem}

\paragraph{Acknowledgement.}
The first author thanks the Japan Society for the Promotion of Science (JSPS), Core-to-Core
Program (A. Advanced Research Networks) for supporting the research.
The second author's research was supported by a grant from the John Templeton Foundation
(``A new dawn of intuitionism: mathematical and philosophical advances," ID 60842).\footnote{The opinions expressed in this publication are those of the authors and do not necessarily reflect the views of the John Templeton Foundation.}


\end{document}